\documentclass[12pt, a4paper]{article}
\usepackage{amsmath}
\usepackage{amsfonts}
\usepackage{dsfont}
\usepackage{mathrsfs}
\usepackage{bbm}
\usepackage{latexsym}
\usepackage{graphicx}
\usepackage{amssymb}

\usepackage{indentfirst}
\usepackage{amsfonts}
\usepackage{amsmath}
\usepackage{epsfig}
\usepackage{amssymb} 
\usepackage{enumerate} 
\usepackage{mathrsfs}
\usepackage{graphicx}
\usepackage{hhline}
\usepackage{array}
\usepackage{cite}
\usepackage{wrapfig}

\newtheorem{theorem}{Theorem}[section]
\newtheorem{lemma}[theorem]{Lemma}

\newtheorem{remark}[theorem]{Remark}

\title{{\bf Sufficient spectral conditions on Hamiltonian and traceable graphs}}

\author{Ruifang Liu$^{a}$\thanks{Supported by the
National Natural Science Foundation of China (No.~11201432) and the
China Postdoctoral Science Foundation (Nos.~2011M501185 and
2012T50636). E-mail address:~rfliu@zzu.edu.cn (R. Liu).}~~~Wai Chee Shiu$^{b}$\thanks{Corresponding author. Supported by Faculty Research Grant of Hong Kong
Baptist University. E-mail address:~wcshiu@hkbu.edu.hk (W.C. Shui).}
~~Jie Xue$^a$
\\ ~ \\
{\footnotesize $^a$ School of Mathematics and Statistics, Zhengzhou
University, Zhengzhou, Henan 450001, China}\\
{\footnotesize $^b$ Department of Mathematics, Hong Kong Baptist University, Kowloon Tong, Hong Kong, China}}
\date{}

\begin{document}
\maketitle

\begin{abstract}
In this paper, we give sufficient conditions on the spectral radius
for a bipartite graph to Hamiltonian and traceable, which expand the
results of Lu, Liu and Tian (2012) \cite{LM}. Furthermore, we also
present tight sufficient conditions on the signless Laplacian
spectral radius for a graph to Hamiltonian and traceable, which
improve the results of Yu and Fan (2012) \cite{YF}.


\end{abstract}

\section{Introduction}

All graphs considered here are simple and undirected. Let $G$ be a
graph with vertex set $V(G)=\{v_{1},v_{2},\ldots,v_{n}\}$ and edge
set $E(G)$. For $v_{i}\in V(G),$ we denote by $d(v_{i})$ or $d_{i}$
the degree of $v_{i}.$ Let $(d_{1},d_{2},\ldots,d_{n})$ be the
degree sequence of $G$, where $d_{1}\leq d_{2}\leq \cdots\leq
d_{n}$. Denote by $\delta(G)$ or simply $\delta$ the minimum degree
of $G$, i.e., $\delta=d_{1}$. The disjoint union of $k$ copies of a
graph $G$ is denoted by $kG$. The {\it join} of $G$ and $H$, denoted
by $G\vee H$, is the graph obtained from disjoint union of $G$ and
$H$ by adding all possible edges between them. Write $K_{n-1}+e$ for
the complete graph on $n-1$ vertices with a pendant edge, and
$K_{n-1}+v$ for the complete graph on $n-1$ vertices together with
an isolated vertex.

A cycle passing through all the vertices of a graph is called a {\it Hamiltonian cycle}. A graph
containing a Hamiltonian cycle is called a {\it Hamiltonian graph}. A path passing through all
the vertices of a graph is called a {\it Hamiltonian path} and a graph containing a Hamiltonian
path is said to be {\it traceable}.

The {\it adjacency matrix} $A(G)=(a_{ij})_{n\times n}$ of a simple
graph $G$ is the matrix indexed by the vertices of $G,$ where
$a_{ij}=1$ if $v_{i}$ is adjacent to $v_{j},$ and $a_{ij}=0$
otherwise. The largest eigenvalue of $A(G),$ denoted by $\rho(G),$
is called the {\it spectral radius} of $G.$ Let $D(G)$ be the degree
diagonal matrix of $G$. The matrix $Q(G)=D(G)+A(G)$ is the {\it
signless Laplacian matrix} of $G.$ Denote by $q(G)$ the {\it
signless Laplacian spectral radius}.

The problem of deciding whether a given graph is Hamiltonian or
traceable is NP-complete. Many reasonable sufficient or necessary
conditions were given for a graph to be Hamiltonian or traceable.
Recently, spectral theory of graphs has been applied to the problem.
Fiedler and Nikiforov \cite{FN} gave sufficient conditions for a
graph to be Hamiltonian and traceable in terms of the spectral
radius of the graph or its complement. Lu, Liu and Tian \cite{LM}
showed a sufficient condition for a graph to be traceable in terms
of the spectral radius of the graph. Subsequently, Zhou \cite{ZB}
investigated the signless Laplacian spectral radius of the
complement of a graph, and provided tight conditions for the
existence of Hamiltonian cycles or paths. Using Laplacian of graphs,
Butler and Chung \cite{FC} established a sufficient condition for a
graph to be Hamiltonian. Fan and Yu \cite{FY} gave a sufficient
condition for a graph to be Hamiltonian with respect to normalized
Laplacian.

For a bipartite graph, Lu, Liu and Tian \cite{LM} gave a sufficient
condition for a bipartite graph being Hamiltonian in terms of the
spectral radius of the quasi-complement of a bipartite graph. In
Section 2, we give sufficient conditions for a bipartite graph to
Hamiltonian and traceable in terms of spectral radius of the
bipartite graph.

Yu and Fan \cite{YF} mentioned the signless Laplacian spectral
conditions for a graph to be Hamiltonian and traceable, while
investigating spectral conditions for a graph to Hamilton-connected.
However, there is a flaw in their result which left out two
exceptional graphs. Hence, it needs further investigate on the
sufficient conditions for a graph to be Hamiltonian and traceable in
terms of the signless Laplacian spectral radius of a graph. In
Section 3, we provide tight sufficient conditions on the signless
Laplacian spectral radius for a graph to Hamiltonian and traceable,
which improve the results of Yu and Fan \cite{YF}.

Note that $\delta\geq 2$ and $\delta\geq 1$ are trivial necessary
conditions for a graph to be Hamiltonian and traceable,
respectively. Hence we always make the assumption while finding
spectral conditions for Hamiltonian and traceable graphs or
bipartite graphs throughout this paper.

\section{Hamiltonian and traceable bipartite graphs}

In this section, we consider the bipartite graphs. Let $G[X,Y]$ be a
bipartite graph. The bipartite graph $G^{\star}[X,Y]$ is called {\it
quasi-complement} of $G$, which is constructed as follows:
$V(G^{\star})=V(G)$ and $xy\in E(G^{\star})$ if and only if
$xy\not\in E(G)$ for $x\in X$, $y\in Y$. Let $K_{n,n-1}$ be a
complete bipartite graph with bipartition $(X,Y)$, where $|X|=n$ and
$|Y|=n-1$. Denote $K_{n,n-1}+e$ the bipartite graph obtained from
$K_{n,n-1}$ by adding a pendent edge to one of vertices in $X$. The
following result was given by Lu {\it et al.} in \cite{LM}.

\begin{theorem}{\bf(\cite{LM})} \label{th5}
Let $G[X,Y]$ be a bipartite graph and $G^{\star}$ the
quasi-complement of $G$, where $|X|=|Y|=n\geq 2$. If
$$\rho(G^{\star})\leq \sqrt{n-1},$$
then $G$ is Hamiltonian unless $G\cong K_{n,n-1}+e$.
\end{theorem}

Theorem~\ref{th5} provided a sufficient condition on Hamiltonian
bipartite graph in terms of spectral radius of the quasi-complement
of the bipartite graph. But there is a minor error in its proof. It
should be:

\begin{theorem}{\bf(\cite{Li})} \label{th5-1}
Let $G[X,Y]$ be a bipartite graph and $G^{\star}$ the
quasi-complement of $G$, where $|X|=|Y|=n\geq 2$. If
$$\rho(G^{\star})\leq \sqrt{\frac{n-2}{2}},$$
then $G$ is Hamiltonian.
\end{theorem}

Next we will show sufficient conditions on Hamiltonian and traceable
bipartite graphs in terms of the spectral radius of bipartite
graphs, respectively. First we state a sharp upper bound on the
spectral radius of a bipartite graph.

\begin{lemma}[\cite{BF}] \label{le13}
If $G$ is a bipartite graph with $m\geq1$ edges and $n$ vertices,
then
$$\rho(G)\leq \sqrt{m},$$
and equality holds if and only if $G\cong K_{p,q}\cup (n-p-q)K_{1}$,
where $pq=m.$
\end{lemma}

A sufficient condition for a bipartite graph to be Hamiltonian was
given in \cite[Ex.~18.3.9]{Bd}.

\begin{lemma}[\cite{Bd}]\label{le14}
Let $G[X,Y]$ be a bipartite graph, where $|X|=|Y|=n\geq 2$, with
degree sequence $(d_{1},d_{2},\ldots,d_{2n})$, where $d_{1}\leq
d_{2}\leq \cdots \leq d_{2n}$. If there is no integer $k\leq n/2$
such that $d_{k}\leq k$ and $d_{n}\leq n-k$. Then $G$ is
Hamiltonian.
\end{lemma}

\begin{lemma}[\cite{LM}] \label{le15}
Let $G[X,Y]$ be a bipartite graph with $\delta\geq1$ and $m$ edges,
where $|X|=|Y|=n\geq 2$. If
$$m\geq n^{2}-n+1,$$
then $G$ is Hamiltonian unless $G\cong K_{n,n-1}+e$.
\end{lemma}

Let $K_{p,n-2}+4e$ be a bipartite graph obtained from $K_{p,n-2}$ by
adding two vertices which are adjacent to two common vertices with
degree $n-2$ in $K_{p,n-2}$, respectively, where $p\ge n-1$. Next we
obtain a result which is similar to Lemma~\ref{le15}.

\begin{lemma} \label{le16}
Let $G[X,Y]$ be a bipartite graph with $\delta\geq 2$ and $m$ edges,
where $|X|=|Y|=n\geq4$. If
$$m\geq n^{2}-2n+4,$$
then $G$ is Hamiltonian unless $G\cong K_{n,n-2}+4e$.
\end{lemma}

\begin{proof}
Suppose that $G$ is a non-Hamiltonian bipartite graph with
$\delta\geq2$ and degree sequence $(d_{1},d_{2},\ldots,d_{2n})$,
where $d_{1}\leq d_{2}\leq \cdots \leq d_{2n}$. By Lemma~\ref{le14},
there exists an integer $k\leq n/2$ such that $d_{k}\leq k$ and
$d_{n}\leq n-k$. Then
\begin{align*}
m&=\frac{1}{2}\sum_{i=1}^{2n}d_{i}\leq \frac{k^{2}+(n-k)^{2}+n^{2}}{2}\\
&=n^{2}-2n+4+(k-2)(k-n+2).
\end{align*}
Since $m\geq n^{2}-2n+4,$ $(k-2)(k-n+2)\geq0.$ Note that since
$k\geq d _{k}\geq\delta\geq 2$ and $n-k\geq d_{n}\geq2$,
$(k-2)(k-n+2)\leq 0$. Hence $(k-2)(k-n+2)=0$ and $m= n^{2}-2n+4,$
and all inequalities in the above argument should be equalities.
Then $k=2$ or $k=n-2.$ If $k=2$, then $G$ is a bipartite graph with
$n^{2}-2n+4$ edges and $d_{1}=d_{2}=2$, $d_{3}=\cdots=d_{n}=n-2$ and
$d_{n+1}=\cdots=d_{2n}=n$. This implies $G\cong K_{n,n-2}+4e$. If
$k=n-2$, then $2\leq n-2\leq n/2$ and hence $n=4$. Then $G$ is a
bipartite graph with $12$ edges and degree sequence
$(2,2,2,2,4,4,4,4)$. Thus $G\cong K_{4,2}+4e$. We complete the
proof. \ \ $\Box$
\end{proof}

\begin{theorem}\label{th26}
Let $G[X,Y]$ be a bipartite graph with $\delta\geq 2$, where $|X|=|Y|=n\geq4$. If
$$\rho(G)\geq\sqrt{n^{2}-2n+4},$$
then $G$ is Hamiltonian unless $G\cong K_{n,n-2}+4e$.
\end{theorem}

\begin{proof}
By Lemma \ref{le13}, we have
$$\sqrt{n^{2}-2n+4}\leq\rho(G)\leq\sqrt{m},$$
where $m$ is the number of edges in $G$. Then $m\geq n^{2}-2n+4$. By
Lemma~\ref{le16}, the result follows.\ \ $\Box$
\end{proof}

Let $G[X,Y]$ be a traceable bipartite graph. Then $|X|=|Y|$ or
$|X|=|Y|+1$. These two types will be discussed separately.

\begin{lemma}[\cite{CV}] \label{le17}
Let $G[X,Y]$ be a bipartite graph, where $|X|=|Y|=n\geq 2$. If
$$d(x)+d(y)\geq n+1$$
for every pair of nonadjacent vertices $x\in X$ and $y\in Y$, then
$G$ is Hamiltonian.
\end{lemma}

\begin{lemma}\label{le18}
Let $G[X,Y]$ be a bipartite graph with $\delta\geq 1$ and $m$ edges,
where $|X|=|Y|=n\geq 3$. If
$$m\geq n^{2}-2n+3,$$
then $G$ is traceable.
\end{lemma}

\begin{proof}
Let $d(x_{0})+d(y_{0})$ be the minimum of $d(x)+d(y)$ for every pair
of nonadjacent vertices $x\in X$ and $y\in Y$. If
$d(x_{0})+d(y_{0})\geq n+1$, then by Lemma~\ref{le17}, $G$ is
Hamiltonian, and thus $G$ is traceable. Suppose that
$d(x_{0})+d(y_{0})\leq n$. Then $$m(G-\{x_{0},y_{0}\})\geq
n^{2}-2n+3-n=(n-1)^{2}-(n-1)+1.$$ By Lemma \ref{le15}, we have
$G-\{x_{0},y_{0}\}$ is Hamiltonian or $G-\{x_{0},y_{0}\}\cong
K_{n-1,n-2}+e$.
\begin{enumerate}[{Case }1.]
\item $G-\{x_{0},y_{0}\}$ is Hamiltonian.

Assume that $C=x_{1}y_{1}x_{2}y_{2}\cdots x_{n-1}y_{n-1}x_{1}$ is a
Hamiltonian cycle in $G-\{x_{0},y_{0}\}.$ Denote by $N(x_{0})$ and
$N(y_{0})$ the neighborhoods of vertices $x_{0}$ and $y_{0}$,
respectively. Let $A=\{y_{i}\; |\; x_{i}\in N(y_{0}), 1\le i\le
n-1\}$ and $B=\{x_{j+1}\; |\; y_{j}\in N(x_{0}), 1\le j \le n-1\}$.
Note that by convention we let $x_n=x_1$.

{\bf Claim: } There exist $y_{s}\in A$ and $x_{t}\in B$ such that
$x_{t}y_{s}\in E(G)$.

Suppose not. Note that $d(x_{0})=|B|$ and $d(y_{0})=|A|$. Then
$$m\leq (n-1)^{2}-|A||B|+|A|+|B|.$$
Since  $$m\geq n^{2}-2n+3,$$ then $(|A|-1)(|B|-1)\leq -1$ which
yields a contradiction. Thus Claim holds.

Let $y_{s}\in A$ and $x_{t}\in B$ such that $x_{t}y_{s}\in E(G)$.
Then $x_{s}y_{0}\in E(G)$ and $x_{0}y_{t-1}\in E(G).$ Without loss
of generality, we can assume that $s\geq t$. Then we find a
Hamiltonian path
$$P=x_{0}y_{t-1}x_{t-1}y_{t-2}\cdots x_{s+1}y_{s}x_{t}y_{t}\cdots y_{s-1}x_{s}y_{0}.$$
Then $G$ is traceable.

\item $G-\{x_{0},y_{0}\}\cong K_{n-1,n-2}+e$.

Let $e=uv$ with $u\in X$ and $v\in Y$, where $v$ is the pendant
vertex in $G-\{x_{0},y_{0}\}$. Hence $d(v)=1$ or $2$ in $G$.

Suppose $d(y_0)\ge 2$. If $d(v)=1$, then $x_0v\notin E(G)$. Thus
$d(x_{0})+d(v)<d(x_{0})+d(y_{0})$ contradicts the minimality of
$d(x_{0})+d(y_{0})$. So $d(v)=2$. In this case, $x_0v\in E(G)$. Let
$u_1y_0\in E(G)$, where $u_1\ne u$. Then $x_0vu$ is the path
starting from $x_0$ and goes into the graph isomorphic to
$K_{n-1,n-2}$. We can travel all vertices of this complete bipartite
graph once and end at $u_1$. Then go to $y_0$.  Hence $G$ is
traceable.

Now, we suppose $d(y_0)=1$. Since
 $n^2-2n+3\le m\le n(n-2) +d(y_0) +d(v)\le n^2-2n+1+d(v)$, $d(v)=2$. Hence $x_0v\in E(G)$.
 Since $|E(K_{n-1,n-2})|=n^2-3n+2\le n^2-2n-1\le m-4$, there are at least $4$ edges being removed when we remove $x_0$, $y_0$ and $v$ from $G$. Since $d(v)=2$, $d(y_0)=1$ and $x_0y_0\notin E(G)$, $d(x_0)\ge 2$.
Without loss of generality, we let $x_1y_0\in E(G)$ and let
$x_0y_1\in E(G)$.

Suppose $x_1=u.$ Then $y_0x_1vx_0y_1$ is a path starting from $y_0$
and goes into the graph isomorphic to $K_{n-2,n-2}$. Hence $G$ is
traceable.

Suppose $x_1\neq u.$ Then $y_0x_1$ is a path starting from $y_0$ and
goes into a graph isomorphic to $K_{n-1,n-2}$. We can travel all
vertices of this complete bipartite graph once and end at $u$. Then
pass $v$ and go to $x_0$. Hence $G$ is traceable.
\end{enumerate}
This completes the proof. \ \ $\Box$
\end{proof}

\begin{theorem}\label{th19}
Let $G[X,Y]$ be a bipartite graph with $\delta\geq 1$, where $|X|=|Y|=n\geq3.$ If
$$\rho(G)\geq\sqrt{n^{2}-2n+3},$$
then $G$ is traceable.
\end{theorem}

\begin{proof}
By Lemma \ref{le13}, we have
$$\sqrt{n^{2}-2n+3}\leq\rho(G)\leq \sqrt{m},$$
then $m\geq n^{2}-2n+3.$ By Lemma \ref{le18}, the theorem holds.\ \
$\Box$
\end{proof}

Next, we consider the other type $|X|=|Y|+1.$ Let $G[X,Y]$ be a
bipartite graph, where $|X|=n+1$ and $|Y|=n\geq2$. Denote by
$\delta_{X}$ and $\delta_{Y}$ the minimum degrees of vertices in $X$
and $Y$, respectively. Note that $\delta_{X}\geq 1$ and
$\delta_{Y}\geq 2$ are the trivial necessary conditions for $G$ to
be traceable. Let $G[X,Y+v]$ be the bipartite graph obtained from
$G[X,Y]$ by adding a vertex $v$ which is adjacent to every vertex in
$X$. It is easy to see that $G[X,Y]$ is traceable if and only if
$G[X,Y+v]$ is Hamiltonian.

Let $K_{n,n-1}+2e$ be a graph obtained from $K_{n,n-1}$ by adding
two vertices which are adjacent to a common vertex with degree
$n-1$, respectively.

\begin{theorem}\label{le20}
Let $G[X,Y]$ be a bipartite graph with $\delta_{X}\geq 1$ and
$\delta_{Y}\geq 2$, where $|X|=n+1$ and $|Y|=n\geq3$. If
$$\rho(G)\geq\sqrt{n^{2}-n+2},$$
then $G$ is traceable unless $G\in\{K_{n+1,n-2}+4e, K_{n,n-1}+2e\}$.
\end{theorem}

\begin{proof} Let $G[X,Y]$ be a bipartite graph with $m$ edges.
Let $G[X,Y]$ be a bipartite graph with $m$ edges. By Lemma
\ref{le13}, we have
$$\sqrt{n^{2}-n+2}\leq\rho(G)\leq\sqrt{m}.$$
Hence $m\geq n^{2}-n+2$. Note that $d(v)=n+1$ in $G[X,Y+v]$, hence
$$m(G[X,Y+v])=m+(n+1)\geq n^{2}+3=(n+1)^{2}-2(n+1)+4.$$
By Lemma \ref{le16}, $G[X,Y+v]$ is Hamiltonian or $G[X,Y+v]\cong
K_{n+1,n-1}+4e$. Hence $G[X,Y]$ is traceable or $G\cong
K_{n+1,n-2}+4e$ or $K_{n,n-1}+2e$. So we have the theorem.\ \ $\Box$
\end{proof}

\section{Hamiltonian and traceable graphs}

In \cite{ZB}, Zhou gave a sufficient condition for a graph to be
Hamiltonian and traceable in terms of the signless Laplacian
spectral radius of the complement of a graph.

Let $\mathbb{EC}_{n}$ be the set of graphs of the following two types of graphs on $n$ vertices: (a) the join of a
trivial graph and a graph consisting of two complete components, and (b) the join of a regular graph
of degree $\frac{n-1}{2}-r$ and a graph on $r$ vertices, where $1\leq r\leq\frac{n-1}{2}.$

Let $\mathbb{EP}_{n}$ be the set of graphs of the following three types of graphs on $n$ vertices: (a) a regular graph
of degree $\frac{n}{2}-1,$ (b) a graph consisting of two complete components, and (c) the join of a regular graph
of degree $\frac{n}{2}-1-r$ and a graph on $r$ vertices, where $1\leq r\leq \frac{n}{2}-1.$

\begin{theorem}{\bf (\cite{ZB})}\label{th1}
Let $G$ be a graph on $n$ vertices with complement $\bar{G}.$

\noindent(i) If $n\geq 3,$ $q(\bar{G})\leq n-1$ and $G\notin\mathbb{EC}_{n},$ then $G$ is Hamiltonian.

\noindent(ii) If $q(\bar{G})\leq n$ and $G\notin\mathbb{EP}_{n},$ then $G$ is traceable.
\end{theorem}

Yu and Fan \cite{YF} mentioned a sufficient condition for a graph to be Hamiltonian and traceable in terms of the signless Laplacian spectral radius of the graph. However, there is a flaw in their result which left out two exceptional graphs $K_{2}\vee 3K_{1}$ and $K_{1,3}$. The complete result is as follows.

\begin{theorem}{\bf (\cite{YF})}\label{th2}
Let $G$ be a graph of order $n\geq3.$

\noindent(i) If $q(G)> 2n-4$ and $G$ is neither $K_{2}\vee 3K_{1}$
nor $K_{n-1}+e$, then $G$ is Hamiltonian.

\noindent(ii) If $q(G)\geq2n-4$ and  $G$ is neither $K_{1,3}$ nor
$K_{n-1}+v$, then $G$ is traceable.
\end{theorem}

In this section, we present new spectral conditions for a graph to be Hamiltonian and traceable in terms of the signless Laplacian spectral
radius of the graph, which improve the results in Theorem \ref{th2}.

The following better sharp upper bound on the signless Laplacian spectral radius for a connected graph $G$ was given in \cite{FLH}, also see \cite{CD}.

\begin{lemma}{\bf (\cite{FLH})} \label{le1}
Let $G$ be a connected graph with $n$ vertices and $m$ edges. Then
$$q(G)\leq \frac{2m}{n-1}+n-2,$$
with equality if and only if $G$ is $K_{1,n-1}$ or $K_{n}$.
\end{lemma}

If $G$ is disconnected, by considering a connected component of $G$,
Yu and Fan \cite{YF} obtained the following sharp upper bound on the
signless Laplacian spectral radius for a general graph $G.$

\begin{lemma}{\bf (\cite{YF})}\label{le2}
Let $G$ be a graph of order $n$ with $m$ edges. Then
$$q(G)\leq \frac{2m}{n-1}+n-2.$$
If $G$ is connected, then the equality holds if and only if $G$ is $K_{1,n-1}$ or $K_{n}$. Otherwise,
the equality holds if and only if $G$ is $K_{n-1}+v$.
\end{lemma}

A sufficient condition for a graph to be Hamiltonian was given by Chv\'{a}tal in 1972.

\begin{lemma}{\bf (\cite{CV})} \label{le3}
Let $G$ be a simple graph with the degree sequence
$(d_{1},d_{2},\ldots, d_{n})$, where $d_{1}\leq d_{2}\cdots\leq
d_{n}$ and $n\geq3$. Suppose that there is no integer $k<n/2$ such
that $d_{k}\leq k$ and $d_{n-k}\leq n-k-1$. Then $G$ is Hamiltonian.
\end{lemma}

Let $\mathbb{N}\mathbb{C}=\{K_{4}\vee 5K_{1}, K_{2}\vee
(K_{3}+2K_{1}), K_{3}\vee 4K_{1}, K_{1,2}\vee 4K_{1}, K_{2}\vee
(K_{1}+K_{1,3}), K_{2}\vee (K_{2}+2K_{1}), K_{1}\vee 2K_{2},
K_{2,3}, K_{2}\vee 3K_{1} \}$ be the set of some graphs. Obviously,
the graphs in $\mathbb{N}\mathbb{C}$ are non-Hamiltonian.

A stronger version of Lemma \ref{le4} occurs in \cite{NB}. In order
to keep this paper complete, independent and self-contained, we
provide a detailed proof again.

\begin{lemma}\footnote{See \cite{NB} for a stronger result.}\label{le4}
Let $G$ be a graph with $\delta\geq 2$. If
$$m>\frac{n^{2}-4n+6}{2},$$ then $G$ is Hamiltonian unless $G\in
\mathbb{N}\mathbb{C}$.
\end{lemma}

\begin{proof}
Let $G$ be a graph on $n$ vertices and $m$ edges with $\delta\geq
2$. Let $d_{1}\leq d_{2}\leq \cdots\leq d_{n}$ be its degree
sequence. Suppose that $G$ is a non-Hamiltonian graph. By Lemma
\ref{le3}, there exists an integer $k<\frac{n}{2}$ such that
$d_{k}\leq k$ and $d_{n-k}\leq n-k-1$. Then
\begin{align}
2m&=\sum_{i=1}^{n}d_{i}=\sum_{i=1}^{k}d_{i}+\sum_{i=k+1}^{n-k}d_{i}+\sum_{i=n-k+1}^{n}d_{i}\nonumber\\
&\leq k^{2}+(n-2k)(n-k-1)+k(n-1)\label{eq-1}\\
&=n^{2}-4n+6+3k^{2}+(1-2n)k+3n-6=n^{2}-4n+6+f(k),\nonumber
\end{align}
where $f(k)=3k^{2}+(1-2n)k+3n-6$. Moreover, since
$n>2k\geq2d_{k}\geq2\delta\geq4$, $n\geq 5$ and $k\ge 2$. Since
$m>\frac{n^{2}-4n+6}{2}$, $f(k)>0$.

The roots of $f(k)$ are $r_1=\frac{1}{6}(2n-1-\sqrt{4n^2-40n+73})$
and  $r_2=\frac{1}{6}(2n-1+\sqrt{4n^2-40n+73})$. Since $f(k)>0$,
either $2\le k<r_1$ or $\frac{n}{2}>k>r_2$.

For $2\le k<r_1$, we have $2n-13>\sqrt{4n^2-40n+73}$. It is easily
to get $n<8$.

For $\frac{n}{2}>k>r_2$, we have $\frac{n-1}{2}> r_2$ for odd $n$.
Then we will get that $n^2-12n+23<0$. So we have $2<n<10$.
Similarly, for even $n$, we will get that $2<n<8$.  Combining with
the lower bound on $n$ obtained above, finally we have $n=9,7,6,5$.

\begin{enumerate}[{Case }1.]
\item $n=9$. We have $k\in\{2,3,4\}$. Since $f(2)=-1$, $f(3)=-2$ and $f(4)=1$, we have $k=4$. That is, $d_{4}\leq 4$ and $d_{5}\leq4$. Note that $51<\sum_{i=1}^{9}d_{i}=2m\leq 52$. So $\sum_{i=1}^{9}d_{i}=52$ and the equality in Eq.~(\ref{eq-1}) holds. That is, the degree sequence of $G$ is $(4,4,4,4,4,8,8,8,8)$. Hence $G\cong K_{4}\vee 5K_{1}$.

 \item $n=7$. We have $k\in\{2,3\}$. We have $f(2)=1$ and $f(3)=3$. If $k=2$, then $27<\sum_{i=1}^{7}d_{i}=2m\leq 28$. So $\sum_{i=1}^{7}d_{i}=28$ and the equality in Eq.~(\ref{eq-1}) holds. That is, the degree sequence of $G$ is $(2,2,4,4,4,6,6)$. Hence $G\cong K_{2}\vee (K_{3}+2K_{1})$.

If $k=3$, then $d_{3}\leq 3$ and $d_{4}\leq 3$. Since
$27<\sum_{i=1}^{7}d_{i}=2m\leq 30$, $\sum_{i=1}^{7}d_{i}=30$ or 28.
When $\sum_{i=1}^{7}d_{i}=30$. Similar to the above case we obtain
that the degree sequence of $G$ is $(3,3,3,3,6,6,6)$.  Thus $G\cong
K_{3}\vee 4K_{1}$. If $\sum_{i=1}^{7}d_{i}=28$, then the degree
sequences of $G$ and $G$ are as follows:
$$
\begin{cases}
(3,3,3,3,5,5,6), \text{ then $G\cong K_{1,2}\vee 4K_{1}$};\\
(3,3,3,3,4,6,6), \text{ then $G\cong K_{2}\vee (K_{2}+K_{1,2})$};\\
(2,3,3,3,5,6,6), \text{ then $G\cong K_{2}\vee (K_{1}+K_{1,3})$}.
\end{cases}
$$
\item $n=6$. We have $k=2$. Thus $18<\sum_{i=1}^{6}d_{i}=2m\leq 20$. So $\sum_{i=1}^{6}d_{i}=20$ and the equality in Eq.~(\ref{eq-1}) holds. Then the degree sequence of $G$ is $(2,2,3,3,5,5)$, and hence $G\cong K_{2}\vee (K_{2}+2K_{1})$.
\item $n=5$. We have $k=2$. Thus $11<\sum_{i=1}^{5}d_{i}=2m\leq 14$, that is, $\sum_{i=1}^{5}d_{i}=12$ or 14. Since $d_{2}\leq 2$ and $d_{3}\leq 2$, we obtain the degree sequences of $G$ and $G$ are:
$$
\begin{cases}
(2,2,2,2,4), \text{ then $G\cong K_{1}\vee 2K_{2}$};\\
(2,2,2,3,3), \text{ then $G\cong K_{2,3}$};\\
(2,2,2,4,4), \text{ then $G\cong K_{2}\vee 3K_{1}$}.
\end{cases}
$$
\end{enumerate}

Note that $K_{2}\vee (K_{2}+K_{1,2})$ contains a Hamiltonian cycle
and the others obtained graphs are nonhamiltonian. Thus the proof is
completed.\ \ $\Box$
\end{proof}

By Lemma \ref{le4},  we present one of the main results.

\begin{theorem}\label{th3}
Let $G$ be a graph on $n\geq4$ vertices with $\delta\geq 2$. If $$q(G)\geq 2n-5+\frac{3}{n-1},$$ then $G$ is Hamiltonian unless $G\in \{K_{3}\vee 4K_{1}, K_{2}\vee 3K_{1} \}$.
\end{theorem}

\begin{proof}
Suppose that $G$ is a non-Hamiltonian graph with $m$ edges.
Obviously, $K_{n}$ is Hamiltonian, $\delta(K_{1,n-1})=1$ and
$\delta(K_{n-1}+v)=0.$ By Lemma \ref{le2}, we have
$q(G)<\frac{2m}{n-1}+n-2.$ Since $q(G)\geq 2n-5+\frac{3}{n-1},$
$m>\frac{n^{2}-4n+6}{2}.$ By Lemma~\ref{le4}, $G\in\mathbb{NC}$. By
directed calculation (see Table~1 at the next page of this paper),
we have $q(G)<2n-5+\frac{3}{n-1}$ for the graphs in $\mathbb{NC}$
except $K_{3}\vee4K_{1}$ and $K_{2}\vee 3K_{1}.$ This completes the
proof. \ \ $\Box$
\end{proof}

The following equivalent condition for a graph to be Hamiltonian and traceable is an exercise in \cite{Bd} (see Ex. 18.1.6).

\begin{lemma}{\bf (\cite{Bd})} \label{le5}
Let $G$ be a graph. Then $G$ is traceable if and only if $G\vee K_{1}$ is Hamiltonian.
\end{lemma}

Let $\mathbb{N}\mathbb{P}=\{K_{3}\vee 5K_{1}, K_{1}\vee
(K_{3}+2K_{1}), K_{2}\vee 4K_{1}, K_{2,4}, K_{1}\vee
(K_{1}+K_{1,3}), K_{1}\vee (K_{2}+2K_{1}), 2K_{2}, K_{1,3} \}$ be
the set of some graphs. Obviously, the graphs in
$\mathbb{N}\mathbb{P}$ are nontraceable. By Lemmas \ref{le4} and
\ref{le5}, we obtain a sufficient condition for a graph to be
traceable.
\begin{equation*}\begin{split}
\text{Table 1: The signless Laplacian spectral radius of some graphs }\\
\begin{tabular}{c|c|c|c}
\hline
$G$ & $q(G)$ & $G$ & $q(G)$\\
\hline
$K_{4}\vee 5K_{1}$&13.1789&$K_{3}\vee 5K_{1}$&10.8990\\
\hline
$K_{2}\vee (K_{3}+2K_{1})$&9.3408&$K_{1}\vee (K_{3}+2K_{1})$&6.9095\\
\hline
$K_{3}\vee 4K_{1}$&9.7720&$K_{2}\vee 4K_{1}$&7.4641\\
\hline
$K_{1,2}\vee 4K_{1}$&8.8965&$K_{2,4}$&6.0000\\
\hline
$K_{2}\vee (K_{1}+K_{1,3})$&9.3408&$K_{1}\vee (K_{1}+K_{1,3})$&6.9095\\
\hline
$K_{2}\vee (K_{2}+2K_{1})$&7.7588&$K_{1}\vee (K_{2}+2K_{1})$&5.3234\\
\hline
$K_{1}\vee 2K_{2}$&5.5616&$2K_{2}$&2.0000\\
\hline
$K_{2,3}$&5.0000&$K_{1,4}$&5.0000\\
\hline
$K_{2}\vee 3K_{1}$&6.3723& $K_{1,3}$&4.0000\\
\hline
\end{tabular}
\end{split}\end{equation*}

A stronger result of Lemma \ref{le6} can be found in \cite{NB}.
\begin{lemma} \label{le6}
Let $G$ be a graph with $\delta\geq 1$ and $m$ edges. If
$m>\frac{n^{2}-4n+3}{2}$, then $G$ is traceable unless $G\in
\mathbb{NP}$.
\end{lemma}

\begin{proof}
Since $|V(G\vee K_{1})|=n+1$ and $|E(G\vee
K_{1})|=m+n>\frac{n^{2}-4n+3}{2}+n=\frac{(n+1)^{2}-4(n+1)+6}{2},$ by
Lemma~\ref{le4}, $G\vee K_{1}$ is Hamiltonian unless $G\vee K_{1}\in
\mathbb{NC}$. According to Lemma~\ref{le5}, $G$ is traceable unless
$G\in \mathbb{NP}$. \ \ $\Box$
\end{proof}

By Lemma \ref{le6}, we easily obtain the following result.

\begin{theorem}\label{th4}
Let $G$ be a graph on $n\geq4$ vertices with $\delta\geq 1$. If $$q(G)\geq 2n-5,$$ then $G$ is traceable unless $G\in \{K_{2}\vee 4K_{1}, K_{1}\vee (K_{2}+2K_{1}), K_{1,3}, K_{1,4}\}$.
\end{theorem}

\begin{proof}
By Lemma~\ref{le2} and the hypothesis, we have $2n-5\le q(G)\leq
\frac{2m}{n-1}+n-2,$ where $m$ is the size of $G$.

Suppose the last equality holds.  By Lemma~\ref{le2}, $G\in\{
K_{1,n-1}, K_{n}, K_{n-1}+v\}$. Clearly, $K_n$ is traceable. Since
$\delta(K_{n-1}+v)=0$, it is not a case. If $G=K_{1,n-1}$, then
$q(G)=n$. Under our assumption, we have $n\le 5$. So $G\in
\{K_{1,3}, K_{1,4}\}$.

Now we assume $2n-5\le q(G)< \frac{2m}{n-1}+n-2$. This implies that
$m>\frac{n^{2}-4n+3}{2}$. By Lemma~\ref{le6}, $G$ is traceable
unless $G\in \mathbb{NP}$. For $G\in \mathbb{NP}$, it is easy to
check that only graphs $K_{2}\vee 4K_{1}, K_{1}\vee (K_{2}+2K_{1})$,
and $K_{1,3}$ satisfying the condition $q(G)\geq 2n-5$ (see Table
1).

Combining both cases, we have the theorem.\ \ $\Box$
\end{proof}

For $n\ge 7$, Theorem~\ref{th4} can be considered as a corollary of
the following recent result:

\begin{theorem}[\cite{Yu}] Let $G$ be a connected graph of order $n\ge 4$. If
$$q(G)\ge\frac{2(n-2)^{2}+4}{n-1}(=2n-6+\frac{6}{n-1}),$$ then $G$ is traceable.
\end{theorem}

Theorem \ref{th4} improves the result in Theorem \ref{th2} (ii). As
a corollary of our result, we present the proof of Theorem \ref{th2}
(ii).

{\bf Proof of Theorem \ref{th2} (ii).} If $n=3,$ then the result
obviously holds.

Now we assume that $n\geq4.$ Suppose $\delta\geq1$. Since
$q(G)\geq2n-4>2n-5$, by Theorem~\ref{th4}, $G$ is traceable unless
$G\in \{K_{2}\vee 4K_{1}, K_{1}\vee (K_{2}+2K_{1}), K_{1,3},
K_{1,4}\}$. By Table~1 and $q(G)\geq2n-4$, we have that $G=K_{1,3}$
is the only exceptional case.

Suppose $\delta=0$ and let $d(v)=0$ in $G$. If $G\neq K_{n-1}+v$,
then $G$ is a proper spanning subgraph of $K_{n-1}+v$. Then
$q(G)<q(K_{n-1}+v)=2n-4$. It is not a case. Thus $G=K_{n-1}+v$,
which is not traceable. This completes the proof.\ \ $\Box$

\begin{remark}\label{re1}
In fact, by much more complicated analysis and excluding much more exceptional graphs, similar to the proofs of Lemmas
\ref{le4} and \ref{le6}, we can reduce the number of edges to $m\geq\binom{n-2}{2}+2$ and $m\geq\binom{n-2}{2}$ for a graph
to be Hamiltonian and traceable, respectively. Using the two results, we can obtain the signless Laplacian spectral conditions $q(G)\geq 2n-6+\frac{4}{n-1}$ and $q(G)\geq2n-6$ for a graph to be Hamiltonian and traceable except several specific graphs, respectively.
\end{remark}

\noindent{\bf Acknowledgment}

The authors would like to thank the anonymous referees for valuable
suggestions and corrections which have considerably improved the
presentation of this paper.

\small {

}
\end{document}